\documentclass[runningheads]{llncs}
\usepackage{graphicx}
\usepackage[symbol]{footmisc}

\usepackage{float}
\usepackage{mathtools}
\usepackage{amssymb, amsmath, latexsym}
\usepackage{nicefrac}
\usepackage{hhline}
\usepackage{makecell}
\usepackage{caption}
\usepackage{multirow}

\usepackage{colortbl}
\definecolor{bgcolor}{rgb}{0.8,1,1}
\definecolor{bgcolor2}{rgb}{0.8,1,0.8}
\usepackage{threeparttable}

\usepackage{pifont}
\usepackage[colorinlistoftodos,bordercolor=orange,backgroundcolor=orange!20,linecolor=orange,textsize=scriptsize]{todonotes}
\usepackage{algorithm}\usepackage{algpseudocode}

\allowdisplaybreaks
\graphicspath{ {images/} }
\def \R {\mathbb R}

\newtheorem{assumption}{Assumption}
  
\def\R{\mathbb{R}}

\def\R{\mathbb R}

\def\EE{\mathbb E}

\def\la{\langle}
\def\ra{\rangle}

\newcommand{\eqdef}{\vcentcolon=}

\newcommand{\EndProof}{\begin{flushright}$\square$\end{flushright}}

\def\<#1,#2>{\langle #1,#2\rangle}

\begin{document} 
\title{SARAH-based Variance-reduced Algorithm for Stochastic Finite-sum Cocoercive Variational Inequalities}
\titlerunning{SARAH for Stochastic Cocoercive Variational Inequalities}
%
\author{Aleksandr Beznosikov\inst{1,2}
Alexander Gasnikov\inst{1,3,4}
}
\authorrunning{A. Beznosikov, A. Gasnikov}
%
\institute{Moscow Institute of Physics and Technology, Dolgoprudny, Russia \and
HSE University, Moscow, Russia \and
IITP RAS, Moscow, Russia \and
Caucasus Mathematical Center, Adyghe State University, Maikop, Russia
}
\maketitle              
\begin{abstract}
Variational inequalities are a broad formalism that encompasses a vast number of applications. Motivated by applications in machine learning and beyond, stochastic methods are of great importance. In this paper we consider the problem of stochastic finite-sum cocoercive variational inequalities. For this class of problems, we investigate the convergence of the method based on the SARAH variance reduction technique. We show that for strongly monotone problems it is possible to achieve linear convergence to a solution using this method. Experiments confirm the importance and practical applicability of our approach.

\end{abstract}
\section{Introduction}

In this paper we focus on the following unconstrained variational inequality (VI) problem:
\begin{equation}
    \label{eq:VI}
   \text{Find} ~~ z^* \in \R^d ~~ \text{such that} ~~ F(z^*) = 0,
\end{equation}
where $F: \R^d \to \R^d $ is some operator. This formulation is broad and encompasses many popular classes of tasks arising in practice. The simplest, however, widely encountered example of the VI is the minimization problem: 
$$
\min_{z \in \R^d} f(z).
$$
To represent it in the form \eqref{eq:VI}, it is sufficient to take $F(z) = \nabla f(z)$. As another also popular practical example, we can consider a saddle point or min-max problem:
$$
    \min_{x \in \R^{d_x}} \max_{y \in \R^{d_y}} g(x,y).
$$
Here we need to take $F(z) = [\nabla_x g(x,y), -\nabla_y g(x,y)]$.

From a machine learning perspective, it is interesting not the deterministic formulation \eqref{eq:VI}, but the stochastic one. More specifically, we want to consider the setup with the operator $F(z) = \EE_{\xi \sim \mathcal{D}} \left[F_{\xi}(z)\right]$, where $\xi$ is a random variable, $\mathcal{D}$ is some distribution, $F_{\xi}: \R^d \to \R^d $ is a stochastic operator. But
it is often the case (especially in practical problems) that the distribution $\mathcal{D}$ is unknown, but we have some samples from $\mathcal{D}$. Then, one can replace with a finite-sum  Monte Carlo approximation, i.e.
\begin{equation}
    \label{eq:fs}
    F(z) = \frac{1}{n} \sum\limits_{i=1}^n F_i(z).
\end{equation}
In the case of minimization problems, statements of the form \eqref{eq:VI}+\eqref{eq:fs} are also called empirical risk minimization \cite{shalev2014understanding}. These types of problems arise both in classical machine learning problems such as simple regressions and in complex, large-scale problems such as neural networks \cite{lecun2015deep}. When it comes to saddle point problems, in recent times the so-called adversarial approach has become popular. Here one can highlight Generative Adversarial Networks (GANs) \cite{goodfellow2020generative} and the adversarial training of models \cite{zhu2019freelb,liu2020adversarial}.

Based on the examples mentioned above, it can be noted that for operators of the form \eqref{eq:fs}, computing the full value of $F$ is a possible but not desirable operation, since it is typically very expensive compared to computing a single operator $F_i$. Therefore, when constructing an algorithm for the problem \eqref{eq:VI}+\eqref{eq:fs}, one wants to avoid computing (or compute very rarely) the full $F$ operator. This task can be solved by a stochastic gradient descent (SGD) framework. Currently, stochastic methods for minimization problems already have a huge background \cite{gorbunov2020unified}. 
The first methods of this type were proposed back in the 1950s by \cite{robbins1951stochastic}.
For example, in the most classic variant, SGD could be written as follows:
\begin{equation}
z^{k+1} = z^k - \eta v^k, \label{eq:SGDstep}
\end{equation}
where $\eta > 0$ is a predefined step-size and $v^k = \nabla f_i(z^k)$, where $i\in [n]$ is chosen randomly 
\cite{shalev2011pegasos}. In this case, the variance of $v_t$ is the main source of slower convergence or convergence only to the neighbourhood of the solution \cite{bottou2018optimization,nguyen2019new,khaled2020better}. 

But for minimization problems of the finite-sum type, one can achieve stronger theoretical and practical results compared to the method \eqref{eq:SGDstep}. This requires the use of a variance reduction technique. Recently, many  variance-reduced variants of SGD have been proposed, including SAG/SAGA \cite{schmidt2017minimizing,defazio2014saga,qian2019saga}, SVRG \cite{johnson2013accelerating,allen2016improved,yang2021accelerating},
MISO \cite{mairal2015incremental},
SARAH \cite{nguyen2017sarah,nguyen2021inexact,nguyen2017stochastic,hu2019efficient}, 
SPIDER \cite{fang2018spider}, STORM \cite{cutkosky2019momentum}, PAGE~\cite{li2021page}.
The essence of one of the earliest and best known variance-reduced methods SVRG is to use $v^k = \nabla f_i(z^k) - \nabla f_i(\tilde z) + \nabla f(\tilde z)$, where $i\in [n]$ is picked at random, where $i\in [n]$ is picked at random and the point $\tilde z$ is updated very rarely (hence we do not need to compute the full gradient often). With this type of methods it is possible to achieve a linear convergence to the solution.
But for both convex and non-convex smooth minimization problems, the best theoretical guarantees of convergence are given by other variance-reduced technique SARAH (and its modifications: SPIDER, STORM, PAGE).

In turn, stochastic methods are also investigated for variational inequalities and saddle point problems \cite{juditsky2011solving,gidel2018variational,hsieh2019convergence,mishchenko2020revisiting,hsieh2020explore,beznosikov2020distributed,gorbunov2022stochastic,beznosikov2022unified,beznosikov2022stochastic}, including methods based on variance reduction techniques \cite{palaniappan2016stochastic,chavdarova2019reducing,Yura2021,alacaoglu2021stochastic,kovalev2022optimal,beznosikov2022unified,beznosikov2022stochastic}. Most of these methods are based on the SVRG approach. At the same time, SARAH-based methods have not been explored for VIs. But as we noted earlier, these methods are the most attractive from the theoretical point of view for minimization problems. The purpose of this paper is to partially close the question of SARAH approach for stochastic finite-sum variational inequalities.

\section{Problem setup and assumptions}



\textbf{Notation.} We use $\la x,y \ra \eqdef \sum_{i=1}^nx_i y_i$ to denote standard inner product of $x,y\in\R^d$ where $x_i$ corresponds to the $i$-th component of $x$ in the standard basis in $\R^d$. It induces $\ell_2$-norm in $\R^d$ in the following way $\|x\|_2 \eqdef \sqrt{\la x, x \ra}$.

Recall that we consider the problem \eqref{eq:VI}, where the operator F has the form \eqref{eq:fs}. Additionally, we assume

\begin{assumption}[Cocoercivity] \label{as:Lipsh}
Each operator $F_i$ is $\ell$-cocoercive, i.e. for all $u, v \in \R^d$ we have
\begin{equation}
\label{eq:Lipsh}
\| F_i(u)-F_i(v) \|^2  \leq \ell \langle F_i(u)-F_i(v) , u - v\rangle.
\end{equation}
\end{assumption}
This assumption is somehow a more restricted analogue of the Lipschetzness of $F_i$. For convex minimization problems, $\ell$-Lipschitzness and $\ell$-cocoercivity are equivalent.  Regarding variational inequalities and saddle point problems, see \cite{loizou2021stochastic}.

\begin{assumption}[Strong monotonicity]\label{as:strmon}
The operator $F$ is $\mu$-strongly monotone, i.e. for all $u, v \in \R^d$ we have
\begin{equation}
\label{eq:strmon}
\langle F(u) - F(v); u - v \rangle \geq \mu \| u-v\|^2.
\end{equation}
\end{assumption}
For minimization problems this property means strong convexity, and for saddle point problems strong convexity--strong concavity.

\section{Main part}

For general Lipschitzness variational inequalities, stochastic methods are usually based not on SGD, but on the Stochastic Extra Gradient method \cite{juditsky2011solving}. But due to the fact that we consider cocoercive VIs, it is sufficient to look at SGD like methods for this class of problems. For example, \cite{loizou2021stochastic} considers SGD, \cite{beznosikov2022stochastic} - SVRG. Following this reasoning, we base our method on the original SARAH \cite{nguyen2017sarah}.

\begin{algorithm}[H]
	\caption{\texttt{SARAH} \cite{nguyen2017sarah} for
Stochastic Cocoercive Variational Inequalities}
	\label{alg:sarah}
	\begin{algorithmic}[1]
\State
\noindent {\bf Parameters:}  Stepsize $\gamma>0$, number of iterations $K, S$.\\
\noindent {\bf Initialization:} Choose  $\tilde z^0 \in \R^d$.
\For {$s=1, 2, \ldots, S$ }
    \State $z^0 = \tilde z^{s-1}$
    \State $v^0 = F(z^0)$
    \State $z^1 = z^0 - \gamma v^0$
    \For {$k=1, 2, \ldots, K-1$ }
        \State Sample $i_k$ independently and uniformly from $[n]$
        \State $v^k = F_{i_k} (z^k) - F_{i_k} (z^{k-1}) + v^{k-1}$
        \State $z^{k+1} = z^k - \gamma v^k$
    \EndFor
    \State $\tilde z^s = z^K$
\EndFor
	\end{algorithmic}
\end{algorithm}

Next, we analyse the convergence of this method. Note that we will use the vector $v^K$ in the analysis, but in reality this vector is not calculated by the algorithm. Our proof are heavily based on the original work on SARAH \cite{nguyen2017sarah}. Lemma \ref{lem:1} gives an understanding of how $\| v^k\|^2$ behaves during the internal loop of Algorithm \ref{alg:sarah}.

\begin{lemma} \label{lem:1}
Suppose that Assumptions \ref{as:Lipsh} and \ref{as:strmon} hold. Consider SARAH (Algorithm \ref{alg:sarah}) with $\gamma \leq \tfrac{1}{\ell}$. Then, we have
\begin{align*}
    \EE[\| v^K \|^2] \leq&
    (1 - \gamma \mu)^K \EE[\| F(z^{0}) \|^2].
\end{align*}
\end{lemma}
\begin{proof}
We start the proof with an update for $v^k$:
\begin{align*}
    \| v^k \|^2 
    =& 
    \| v^{k-1} \|^2 + \| F_{i_k} (z^k) - F_{i_k} (z^{k-1}) \|^2 + 2 \langle F_{i_k} (z^k) - F_{i_k} (z^{k-1}) , v^{k-1} \rangle.
\end{align*}
Next, we use an update for $z^k$ and make a small rearrangement
\begin{align*}
    \| v^k \|^2 
    =&
    \| v^{k-1} \|^2 + \| F_{i_k} (z^k) - F_{i_k} (z^{k-1}) \|^2 - \frac{2}{\gamma} \langle F_{i_k} (z^k) - F_{i_k} (z^{k-1}) , z^k - z^{k-1} \rangle
    \\
    =&
    \| v^{k-1} \|^2 + \| F_{i_k} (z^k) - F_{i_k} (z^{k-1}) \|^2 - \frac{1}{\gamma} \langle F_{i_k} (z^k) - F_{i_k} (z^{k-1}) , z^k - z^{k-1} \rangle
    \\
    &- \frac{1}{\gamma} \langle F_{i_k} (z^k) - F_{i_k} (z^{k-1}) , z^k - z^{k-1} \rangle.
\end{align*}
Taking the full mathematical expectation, we obtain
\begin{align*}
    \EE[\| v^k \|^2]
    =&
    \EE[\| v^{k-1} \|^2] + \EE[\| F_{i_k} (z^k) - F_{i_k} (z^{k-1}) \|^2]
    \\
    &- \frac{1}{\gamma} \EE[\langle F_{i_k} (z^k) - F_{i_k} (z^{k-1}) , z^k - z^{k-1} \rangle]
    \\
    &- \frac{1}{\gamma} \EE[\langle F_{i_k} (z^k) - F_{i_k} (z^{k-1}) , z^k - z^{k-1} \rangle].
\end{align*}
Independence of the $i_k$ generation gives
\begin{align*}
    \EE[\| v^k \|^2]
    =&
    \EE[\| v^{k-1} \|^2] + \EE[\| F_{i_k} (z^k) - F_{i_k} (z^{k-1}) \|^2] 
    \\
    &- \frac{1}{\gamma} \EE[\langle F_{i_k} (z^k) - F_{i_k} (z^{k-1}) , z^k - z^{k-1} \rangle]
    \\
    &- \frac{1}{\gamma} \EE[\langle \EE_{i_k}[F_{i_k} (z^k) - F_{i_k} (z^{k-1})] , z^k - z^{k-1} \rangle]
    \\
    =&
    \EE[\| v^{k-1} \|^2] + \EE[\| F_{i_k} (z^k) - F_{i_k} (z^{k-1}) \|^2] 
    \\
    &- \frac{1}{\gamma} \EE[\langle F_{i_k} (z^k) - F_{i_k} (z^{k-1}) , z^k - z^{k-1} \rangle]
    \\
    &- \frac{1}{\gamma} \EE[\langle F (z^k) - F (z^{k-1}) , z^k - z^{k-1} \rangle].
\end{align*}
With Assumptions \ref{as:Lipsh} and \ref{as:strmon}, we get
\begin{align*}
    \EE[\| v^k \|^2]
    \leq&
    \EE[\| v^{k-1} \|^2] + \EE[\| F_{i_k} (z^k) - F_{i_k} (z^{k-1}) \|^2]
    \\
    &- \frac{1}{\gamma \ell} \EE[\| F_{i_k} (z^k) - F_{i_k} (z^{k-1}) \|^2]
    \\
    &- \frac{\mu}{\gamma} \EE[\|z^k - z^{k-1} \|^2]
    \\
    =&
    (1 - \gamma \mu)\EE[\| v^{k-1} \|^2] + \left( \frac{\gamma \ell - 1}{\gamma \ell} \right)\EE[\| F_{i_k} (z^k) - F_{i_k} (z^{k-1}) \|^2].
\end{align*}
In the last step we substitute $z^{k-1} - z^k = \gamma v^k$. The choice of $0 < \gamma \leq \tfrac{1}{\ell}$ gives
\begin{align*}
    \EE[\| v^k \|^2] 
    \leq&
    (1 - \gamma \mu)\EE[\| v^{k-1} \|^2].
\end{align*}
Running recursion and using $v^0 = F(z^0)$, we finish the proof.
\EndProof
\end{proof}

The following lemma gives how different $v^K$ and $F(z^K)$ are in the inner loop of Algorithm \ref{alg:sarah}.

\begin{lemma} \label{lem:2}
Suppose that Assumption \ref{as:Lipsh} holds. Consider SARAH (Algorithm \ref{alg:sarah}). Then, we have
\begin{align*}
    \EE[\| F(z^K) - v^K\|^2] \leq  \frac{\gamma \ell}{2 - \gamma \ell} \EE[\|F(z^0)\|^2].
\end{align*}
\end{lemma}
\begin{proof} Let us consider the following chain of reasoning:
\begin{align*}
    \EE[\| F(z^k) - v^k\|^2] 
    =&
    \EE[\|[F(z^{k-1}) - v^{k-1}]  + [F(z^k) - F(z^{k-1})] - [v^k - v^{k-1}]\|^2]
    \\
    =&
    \EE[\|F(z^{k-1}) - v^{k-1} \|^2] + \EE[\| F(z^k) - F(z^{k-1}) \|^2]
    \\
    &+ \EE[\|v^k - v^{k-1} \|^2]
    \\
    &+2 \EE[\langle F(z^{k-1}) - v^{k-1},  F(z^k) - F(z^{k-1})\rangle]
    \\
    &-2 \EE[\langle F(z^{k-1}) - v^{k-1},  v^k - v^{k-1}\rangle]
    \\
    &-2 \EE[\langle F(z^k) - F(z^{k-1}),  v^k - v^{k-1}\rangle]
    \\
    =&
    \EE[\|F(z^{k-1}) - v^{k-1} \|^2] + \EE[\| F(z^k) - F(z^{k-1}) \|^2]
    \\
    &+ \EE[\|v^k - v^{k-1} \|^2]
    \\
    &+2 \EE[\langle F(z^{k-1}) - v^{k-1},  F(z^k) - F(z^{k-1})\rangle]
    \\
    &-2 \EE[\langle F(z^{k-1}) - v^{k-1},  \EE_{i_k}[v^k - v^{k-1}]\rangle]
    \\
    &-2 \EE[\langle F(z^k) - F(z^{k-1}),  \EE_{i_k}[v^k - v^{k-1}]\rangle]
    \\
    =&
    \EE[\|F(z^{k-1}) - v^{k-1} \|^2] - \EE[\| F(z^k) - F(z^{k-1}) \|^2] 
    \\
    &+ \EE[\|v^k - v^{k-1} \|^2]
    \\
    \leq&
    \EE[\|F(z^{k-1}) - v^{k-1} \|^2] + \EE[\|v^k - v^{k-1} \|^2].
\end{align*}
Here we also use that
$$
\EE_{i_k}[v^k - v^{k-1}] = \EE_{i_k}[F_{i_k} (z^k) - F_{i_k} (z^{k-1})] = F (z^k) - F(z^{k-1}).
$$
Running recursion and using $v^0 = F(z^0)$, we have
\begin{align}
    \label{eq:temp1}
    \EE[\| F(z^K) - v^K\|^2] \leq  \sum_{k=1}^K \EE[\|v^k - v^{k-1} \|^2].
\end{align}
In the same way as in Lemma \ref{lem:1}, we can derive
\begin{align*}
    \| v^k \|^2 
    =& 
    \| v^{k-1} \|^2 + \| F_{i_k} (z^k) - F_{i_k} (z^{k-1}) \|^2 + 2 \langle F_{i_k} (z^k) - F_{i_k} (z^{k-1}) , v^{k-1} \rangle
    \\
    =&
    \| v^{k-1} \|^2 + \| F_{i_k} (z^k) - F_{i_k} (z^{k-1}) \|^2 - \frac{2}{\gamma} \langle F_{i_k} (z^k) - F_{i_k} (z^{k-1}) , z^k - z^{k-1} \rangle
    \\
    \leq&
    \| v^{k-1} \|^2 + \| F_{i_k} (z^k) - F_{i_k} (z^{k-1}) \|^2 - \frac{2}{\gamma \ell} \| F_{i_k} (z^k) - F_{i_k} (z^{k-1}) \|^2
    \\
    =&
    \| v^{k-1} \|^2 + \left( \frac{\gamma \ell - 2}{\gamma \ell}\right) \| F_{i_k} (z^k) - F_{i_k} (z^{k-1}) \|^2
    \\
    =&
    \| v^{k-1} \|^2 + \left( \frac{\gamma \ell - 2}{\gamma \ell}\right) \| v^{k} - v^{k-1} \|^2.
\end{align*}
After a small rewriting and with the full expectation, we get
\begin{align*}
    \EE[\| v^{k} - v^{k-1} \|^2]
    \leq& 
    \frac{\gamma \ell}{2 - \gamma \ell} \EE[\| v^{k-1} \|^2 - \| v^{k} \|^2].
\end{align*}
By substituting this into the expression \eqref{eq:temp1} and using $v^0 = F(z^0)$, we finish the proof.
\EndProof
\end{proof}

Let us combine Lemmas \ref{lem:1} and \ref{lem:2} into the main theorem of this paper. 

\begin{theorem} \label{th:1}
Suppose that Assumptions \ref{as:Lipsh} and \ref{as:strmon} hold. Consider SARAH (Algorithm \ref{alg:sarah}) with $\gamma = \tfrac{2}{9 \ell}$ and $K = \tfrac{10 \ell}{\mu}$. Then, we have
\begin{align*}
    \EE[\| F (\tilde z^s) \|^2] 
    &\leq \frac{1}{2} \EE[\| F(\tilde z^{s-1}) \|^2].
\end{align*}
\end{theorem}
\begin{proof}
We start from 
\begin{align*}
    \EE[\| F (z^K) \|^2] \leq 2 \EE[\| F(z^K) - v^K \|^2] + 2 [\EE\| v^K \|^2].
\end{align*}
Applying Lemma 1 and 2, we have
\begin{align*}
    \EE[\| F (z^K) \|^2] 
    &\leq 
    \left[\frac{2\gamma \ell}{2 - \gamma \ell} + 2 (1 - \gamma \mu)^K \right]\EE[\| F(z^{0}) \|^2]
    \\
    &\leq 
    \left[\frac{2\gamma \ell}{2 - \gamma \ell} + 2 \exp(- \gamma \mu K) \right]\EE[\| F(z^{0}) \|^2].
\end{align*}
Here we also use that $\gamma \mu \in (0;1)$ (for $\gamma \leq \tfrac{2}{9 \ell}$) and then $(1 - \gamma \mu)\leq \exp(-\gamma\mu)$.
The substitution $\gamma$ and $K$ gives
\begin{align*}
    \EE[\| F (z^K) \|^2] 
    &\leq \frac{1}{2} \EE[\| F(z^0) \|^2].
\end{align*}
We know that $z^0 = \tilde z^{s-1}$ and $z^K = \tilde z^s$ and have
\begin{align*}
    \EE[\| F (\tilde z^s) \|^2] 
    &\leq \frac{1}{2} \EE[\| F(\tilde z^{s-1}) \|^2].
\end{align*}
\EndProof
\end{proof}

Since we need to find a point $z$ such that $F (z) \approx F(z^*) = 0$, we can easily get an estimate on the oracle complexity (number of $F_i$ calls) to achieve precision $\varepsilon$.

\begin{corollary}
Suppose that Assumptions \ref{as:Lipsh} and \ref{as:strmon} hold. Consider SARAH (Algorithm \ref{alg:sarah}) with $\gamma = \tfrac{2}{9 \ell}$ and $K = \tfrac{10 \ell}{\mu}$. Then, to achieve $\varepsilon$-solution ($\EE\| F(\tilde z^S)\|^2 \sim \varepsilon^2$), we need
\begin{align*}
    \mathcal{O}\left( \left[n + \frac{\ell}{\mu} \right]\log_2 \frac{\| F(z^0)\|^2}{\varepsilon^2}\right) \quad \text{oracle calls}.
\end{align*}
\end{corollary}
\begin{proof}
From Theorem \ref{th:1} we need the following number of outer iterations:
\begin{align*}
    S = \mathcal{O}\left( \log_2 \frac{\| F(z^0)\|^2}{\varepsilon^2}\right).
\end{align*}
At each outer iteration we compute the full operator one time, and at the remaining $K-1$ iterations we call the single operator $F_i$ two times per one inner iteration. Then, the total number of oracle calls is
\begin{align*}
    S \times (2 \times (K-1) + n) = \mathcal{O}\left( \left[n + \frac{\ell}{\mu} \right]\log_2 \frac{\| F(z^0)\|^2}{\varepsilon^2}\right).
\end{align*}
\EndProof
\end{proof}

Note that the obtained oracle complexity coincides with the similar complexity for SVRG from \cite{beznosikov2022stochastic}. It is interesting to see how these methods behave in practice.

\section{Experiments}

The aim of our experiments is to compare the performance of different methods for stochastic finite-sum cocoercive variational inequalities. In particular, we use SGD from \cite{loizou2021stochastic}, SVRG from \cite{beznosikov2022stochastic} and SARAH. We conduct our experiments on a finite-sum bilinear saddle point problem:
\begin{align}
    \label{bilinear}
     g(x,y) = \frac{1}{n} \sum\limits_{i=1}^n \left[g_i(x,y) = x^\top A_i y + a^\top_i x + b^\top_i y + \frac{\lambda}{2} \| x\|^2 -  \frac{\lambda}{2} \|y\|^2\right],
\end{align}
where $A_i \in \R^{d \times d}$, $a_i, b_i \in \R^d$. This problem is $\lambda$-strongly convex--strongly concave and, moreover, $L$-smooth with $ L = \|A \|_2$ for $A = \tfrac{1}{n} \sum_{i=1}^n A_i$. We take $n = 10$, $d=100$ and generate matrix $A$ and vectors $a_i, b_i$ randomly, $\lambda = 1$. For this problem the cocoercivity constant $\ell = \tfrac{\|A \|_2^2}{\lambda}$. The steps of the methods are selected for best convergence. For SVRG and SARAH the number of iterations for the inner loops is taken as $\tfrac{\ell}{\lambda}$. We run three experiment setups: with small $\ell \approx 10^2$, medium  $\ell \approx 10^3$ and big $\ell \approx 10^4$.

\begin{figure}[h]
\begin{minipage}{0.33\textwidth}
  \centering
\includegraphics[width =  \textwidth ]{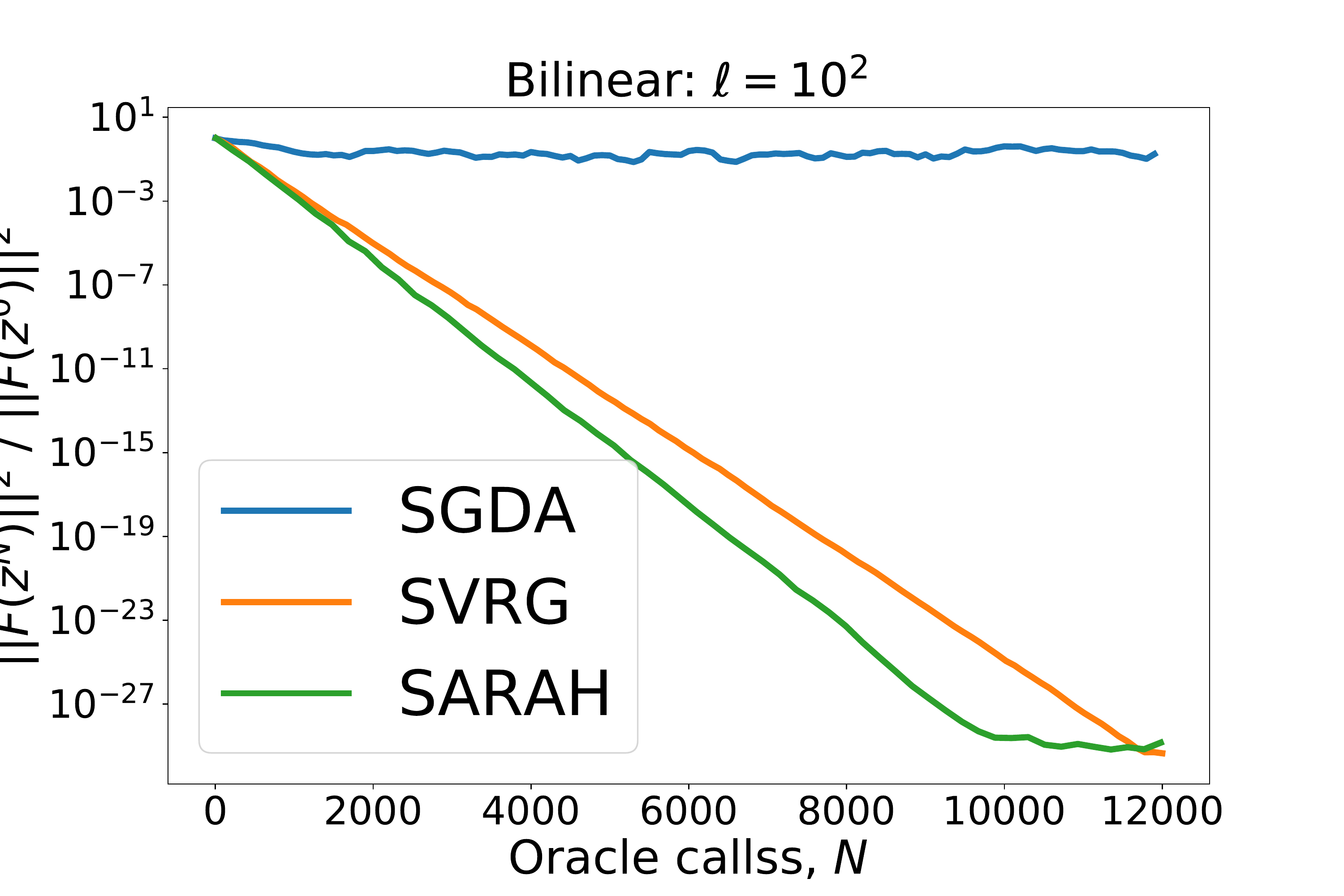}
\end{minipage}%
\begin{minipage}{0.33\textwidth}
  \centering
\includegraphics[width =  \textwidth ]{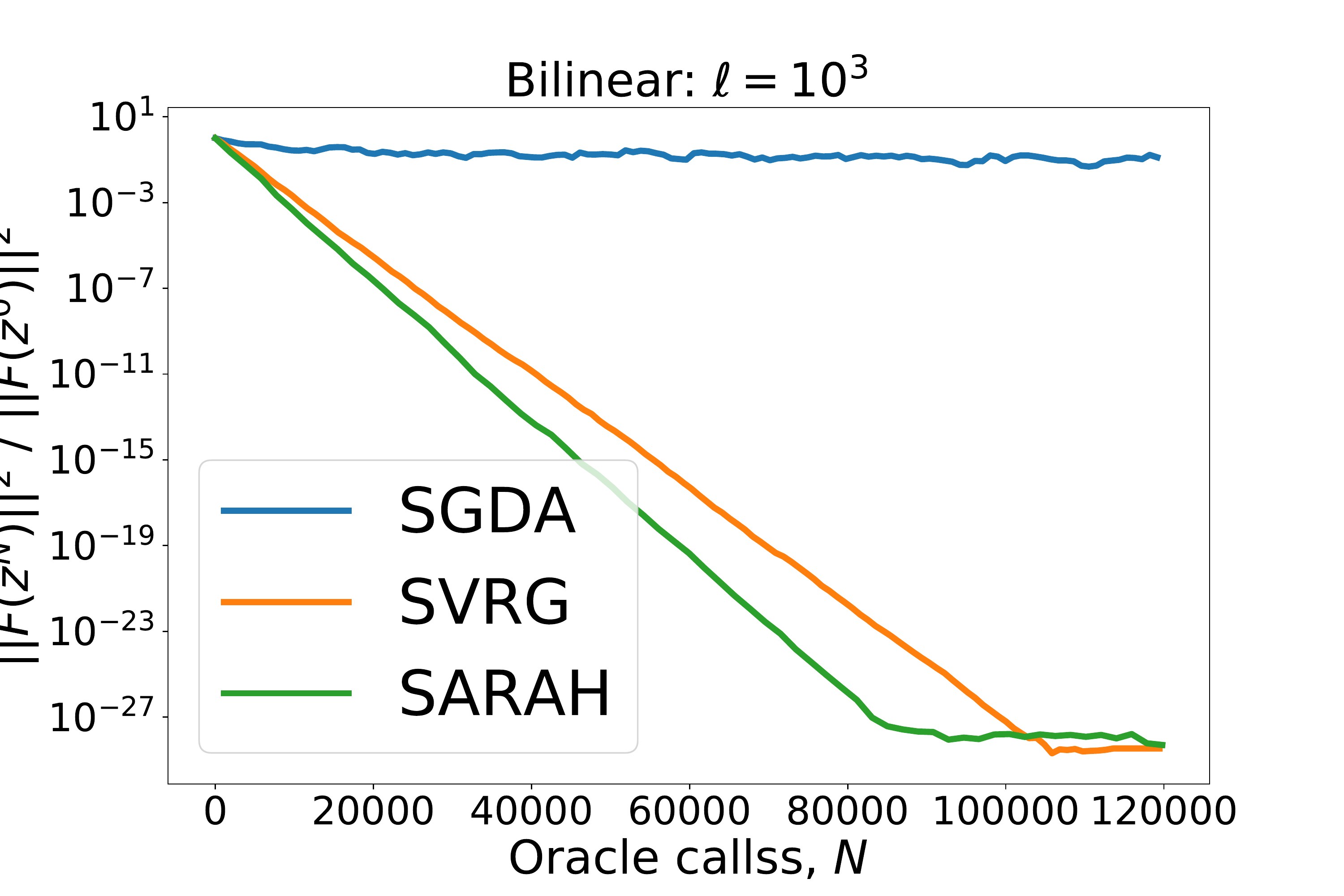}
\end{minipage}%
\begin{minipage}{0.33\textwidth}
  \centering
\includegraphics[width =  \textwidth ]{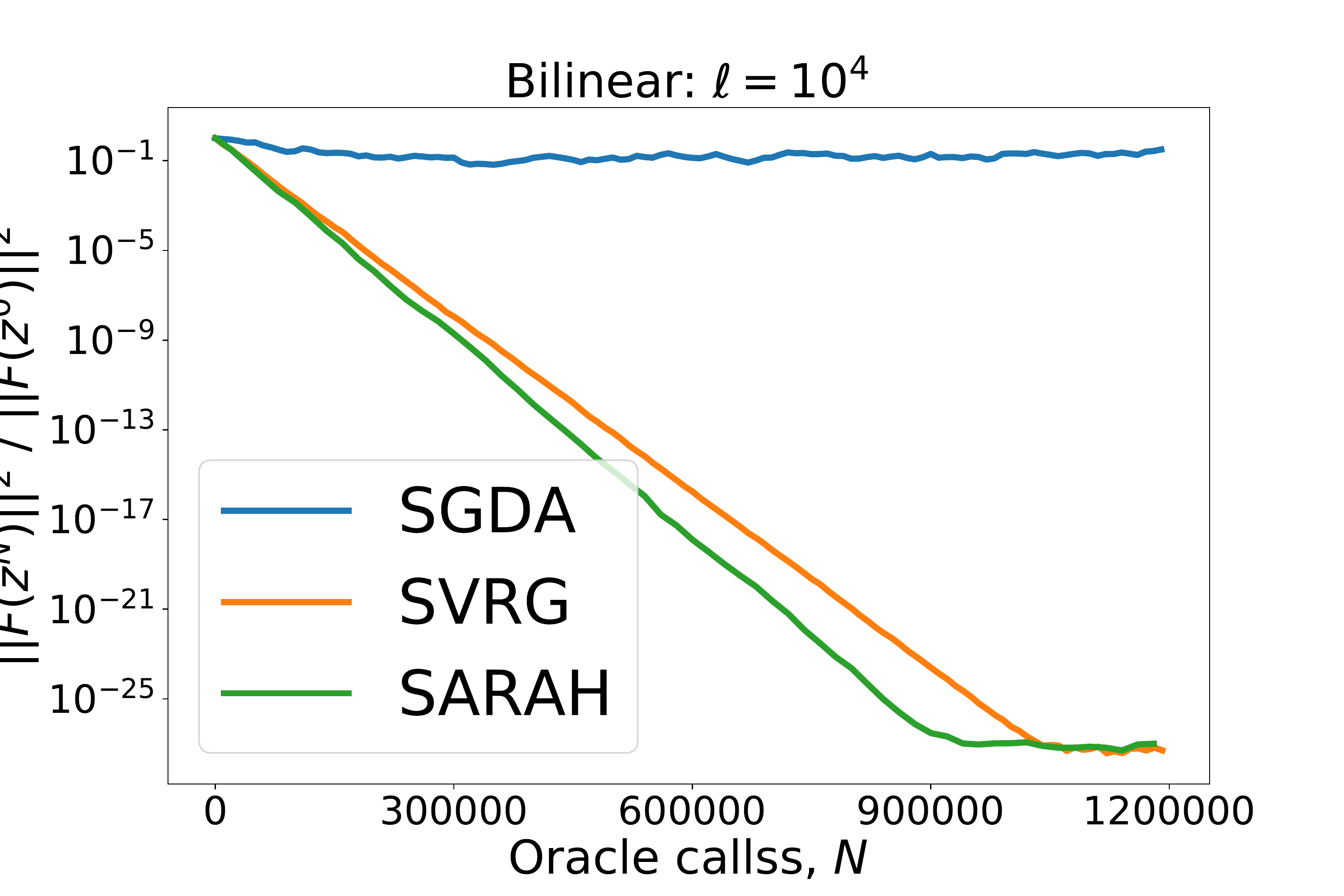}
\end{minipage}%
\\
\begin{minipage}{0.33\textwidth}
  \centering
(a) small $\ell$
\end{minipage}%
\begin{minipage}{0.33\textwidth}
\centering
 (b) medium $\ell$
\end{minipage}%
\begin{minipage}{0.33\textwidth}
\centering
  (c) big $\ell$
\end{minipage}%
\caption{\small Bilinear problem  \eqref{bilinear}: Comparison of state-of-the-art SGD-based methods for stochastic cocoercive VIs.}
    \label{fig:min}
\end{figure}
See Figure \ref{fig:min} for the results. We see that SARAH converges better than SVRG, and SGD converges much slower.


\section*{Acknowledgments}

The authors would like to congratulate Boris Mirkin on his jubilee and wish him good health and new scientific advances.

The work of A. Beznosikov was supported by the strategic academic leadership program 'Priority 2030' (Agreement  075-02-2021-1316 30.09.2021).
The work of A. Gasnikov was supported by the Ministry of Science and Higher Education of the Russian Federation (Goszadaniye), No. 075-00337-20-03, project No. 0714-2020-0005.

\bibliographystyle{splncs04}
\bibliography{ltr}

\newpage 

\appendix

\end{document}